\documentclass[12pt,reqno]{amsart}
\usepackage{hyperref}
\usepackage{amsmath,amssymb}
\numberwithin{equation}{section}

\usepackage{float}
\usepackage{pgf,tikz}

\newcommand \cochord{\operatorname{co-chord}}
\newcommand \minmax{\operatorname{min-max}}
\newcommand \reg{\operatorname{reg}}
\newcommand\ma{\operatorname{m}}
\newcommand\link{\operatorname{link}}

\newcommand\PP{\mathcal{P}}

\newcommand \X{\mathcal{X}}

\newcommand \K{\mathbb{K}}
\newcommand \s{\mathcal{S}}

\newtheorem{theorem}{Theorem}[section]
\newtheorem{definition}[theorem]{Definition}
\newtheorem{lemma}[theorem]{Lemma}
\newtheorem{proposition}[theorem]{Proposition}

\newtheorem{obs}[theorem]{Observation}
\newtheorem{question}[theorem]{Question}
\newtheorem{remark}[theorem]{Remark}
\newtheorem{conj}[theorem]{Conjecture}
\newtheorem{corollary}[theorem]{Corollary}

\newtheorem*{notation*}{Notation}
\newtheorem{notation}[theorem]{Notation}
\begin{document}

\title[Upper bounds for the regularity of powers of edge ideals of graphs]
{Upper bounds for the regularity of powers of edge ideals of graphs}
\author{A. V. Jayanthan}
\email{jayanav@iitm.ac.in}
\address{Department of Mathematics, Indian Institute of Technology
Madras, Chennai, INDIA - 600036}
\author{S. Selvaraja}
\email{selva.y2s@gmail.com}
\address{Chennai Mathematical Institute, H1, SIPCOT IT Park, Siruseri, Kelambakkam, Chennai,
INDIA - 603103.}
\thanks{AMS Classification 2010: 13D02, 13F20, 05C25}

  \maketitle

\begin{abstract}
Let $G$ be a finite simple graph and $I(G)$ denote the corresponding
edge ideal. In this paper, we obtain upper bounds for the
Castelnuovo-Mumford regularity of $I(G)^q$
in terms of certain combinatorial invariants associated with $G$. We also prove
a weaker version of a conjecture by Alilooee, Banerjee,
Beyarslan and H\`a on an upper bound for the
regularity of $I(G)^q$ and we prove the
conjectured upper bound for the class of vertex decomposable graphs. Using these 
results, we explicitly compute the regularity of $I(G)^q$ for several classes of graphs. 
\end{abstract}

\section{Introduction}
Let $I$ be a homogeneous ideal of a polynomial ring 
$R = \K[x_1,\ldots,x_n]$ over a field $\K$ with usual grading.
In \cite{BEL91}, Bertram, Ein and Lazarsfeld  
have initiated  the study of the Castelnuovo-Mumford regularity,
henceforth denoted as $\reg(-)$, of $I^q$
as  a function of $q$ by  proving that if $I$ is the defining ideal of a smooth complex 
projective variety, then $\reg(I^q)$ is bounded by a linear function of $q$. 
Then, Chandler \cite{Cha97} and Geramita, Gimigliano and Pitteloud \cite{GGP95} proved that
if $\dim(R/I)\leq 1$, then $\reg(I^q) \leq q\reg(I)$ for all $q\geq 1$. 
However, Swanson \cite{Swa97} proved that there exists $k \geq 1$ such that
for all $q \geq 1$, $\reg(I^q)\leq kq$. Thereafter, Cutkosky, Herzog and Trung, \cite{CHT}, 
and independently Kodiyalam \cite{vijay}, proved that for a homogeneous ideal 
$I$ in a polynomial ring,  $\reg(I^q)$ is a linear 
function for $q \gg 0$ i.e., 
there exist non negative integers $a$ and $b$ depending on $I$ such that
$\reg(I^q)=aq+b \text{ for all $q \gg 0$}.$
While the coefficient $a$ is well-understood (\cite{CHT}, \cite{vijay}, \cite{TW}), 
the free constant $b$ and the stabilization index $q_0=\min\{q' \mid \reg(I^q)=aq+b, 
\text{ for all } q \geq q'\}$ are quite mysterious.
Therefore, the attention has been to identify classes for which the linear
polynomial can be computed or bounded using invariants associated to $I$.
There have been some attempts on computing the free constant and stabilization index
for several class of ideals. For instance, if $I$  is a equigenerated homogeneous ideal, then 
$b$ is related to the regularity of fibers of certain projection map 
(see for example, \cite{Romer2001}). If
$I$ is $(x_1,\ldots,x_n)$-primary, then $q_0$ can be related to partial
regularity of the Rees algebra of $I$ (see for example, \cite{Berle}).
In this paper, we study the regularity
of powers of edge ideals associated to finite simple graphs.

Let $G$ be a finite simple graph without isolated vertices on the vertex set $\{x_1, \ldots,
x_n\}$ and $I(G) := (\{x_ix_j \mid \{x_i,x_j\} \in E(G)\}) \subset \K[x_1,
\ldots, x_n]$ be the edge ideal corresponding to the graph $G$.
It is known that
$\reg(I(G)^q) = 2q + b$ for some $b$ and $q \geq q_0$. 
There are very few classes of graphs for which $b$ and $q_0$ are known. We
refer the reader to \cite{BBH17} and the references cited there for a
review of results in the literature in this direction.
While the aim is to obtain the linear polynomial corresponding to
$\reg(I(G)^q)$, it seems unlikely that a single combinatorial invariant
will represent the constant term for all graphs. This naturally give
rise to two directions of research. One direction is to obtain linear polynomials
for particular classes of graphs. Another direction is to obtain upper and
lower bounds for $\reg(I(G)^q)$ using combinatorial invariants
associated to the graph $G$. It was proved by Beyarslan, H\`a and Trung that $2q + \nu(G) -
1 \leq \reg(I(G)^q)$ for all $q \geq 1$, where $\nu(G)$ denotes the
induced matching number of $G$, \cite{selvi_ha}. In \cite{jayanthan},
the authors along with Narayanan proved that for a bipartite graph
$G$, $\reg(I(G)^q) \leq 2q + \cochord(G) - 1$ for all $q \geq 1$,
where $\cochord(G)$ denote the co-chordal cover number of $G$. 
There is no general upper bound known for powers of edge ideals of
arbitrary graphs. Therefore, one may ask:
\begin{enumerate}
	\item[Q1.] Does there exists a function $\rho: \{\text{ finite simple
		graphs }\} \rightarrow \mathbb{N}$ such that for any given graph $G$, 
    $\reg(I(G)^q) \leq 2q + \rho(G)$ for all $q \geq 1$.
 \item[Q2.] Can one obtain the linear polynomial corresponding to
   $\reg(I(G)^q)$ for various classes of graphs?
\end{enumerate}
This paper evolves around these two questions.

The first main result of the paper answers Question Q1. 
We prove that if the numerical function $\rho$ satisfies certain
properties, then an upper bound as in Q1 is true:

We first fix a notation that we consider throughout this paper.
Let $G$ be a graph and  
$\mathcal{I}_G$ be the set of all non-empty induced subgraphs of $G$.
\vskip 1mm \noindent
\textbf{Theorem \ref{main-result}.}
{\em
Let $G$  be a graph and  $\rho: \mathcal{I}_G \longrightarrow 
\mathbb{N}$ be a function such that for any $L \in \mathcal{I}_G$,
\begin{enumerate}
 \item $\reg(I(L)) \leq \rho(L)+1$,
 \item $\rho(L_1) \leq \rho(L)$ for any induced subgraph $L_1$ of $L$
   and
 \item there exists a vertex $x \in V(L)$ such that $\rho(L \setminus N_L[x])+1 \leq \rho(L)$.
\end{enumerate}
Then }
\[
 \reg(I(G)^q) \leq 2q+\rho(G)-1 ~\text{ for all $q \geq 1$.}
\]

As an application of this result, we obtain upper
bounds in terms of certain specific combinatorial invariants.
H\`a and Woodroofe
\cite{HaWood} defined an invariant in terms of star
packing, denoted by $\zeta(G)$ (see Section \ref{reg-pow} for the
definition), and proved that $\reg(I(G)) \leq \zeta(G) + 1$.  Also,
Woodroofe proved that $\reg(I(G)) \leq \cochord(G) + 1$. Alilooee, Banerjee,
Beyarslan and H\'{a}  conjectured
(\cite[Conjecture 7.11(1)]{BBH17}):

\begin{conj}\label{cochord-conj}
Let $G$ be a graph. Then for all $q \geq 1$, $\reg(I(G)^q) \leq
2q+\cochord(G) -1$.
\end{conj}

In this paper, we prove Conjecture \ref{cochord-conj}. We also extend the result
by H\`a and Woodroofe  to all powers.

\vskip 2mm \noindent
\textbf{Theorem \ref{mainresult}.}
{\em 
Let $G$  be a graph.  Then for all $q \geq 1$,
\begin{enumerate}
 \item $\reg(I(G)^q) \leq 2q+\zeta(G)-1.$
 \item $\reg(I(G)^q) \leq 2q+\cochord(G)-1$.
\end{enumerate}
}
Another way of bounding the function $\reg(I(G)^q)$, than using
combinatorial invariants, is to relate it to the regularity of $G$
itself. It was conjectured by Alilooee, Banerjee, Beyarslan and H\`a,
\cite[Conjecture 7.11(2)]{BBH17}: 

\begin{conj}\label{ABBH-conj}
If $G$ is a graph, then for all $q \geq 1$,
$\reg(I(G)^q) \leq 2q + \reg(I(G)) - 2.$
\end{conj}
There are some classes of graphs for which this conjecture is known to
be true, see \cite{BBH17, BN19}. As a consequence of the techniques that we have
developed, we prove the conjecture with an additional hypothesis:
\vskip 2mm \noindent
\textbf{Corollary \ref{weak-conj1}.} {\em
Let $G$ be a graph. If every induced subgraph $H$ of 
$G$ has a vertex $x$ with $\reg(I(H \setminus N_H[x]))+1 \leq
\reg(I(H))$, then for all $q \geq 1$,}
\[
 \reg(I(G)^q) \leq 2q+\reg(I(G))-2.
\]

We  recover many of the known results on the regularity of powers of
edge ideals of graphs (Corollary \ref{known-results}).  
Also, as a
consequence of our results we answer Q2 by obtaining precise
expressions for the regularity of powers of edge ideals of some
classes of graphs, 
(Proposition \ref{mat-reg}, Proposition \ref{disjoint}).

So far, in the literature, for the
classes of graphs for which the regularity of powers of edge ideals
have been computed, they satisfy either $\reg(I(G)^q) = 2q + \nu(G) -
1$ or $\reg(I(G)^q) = 2q + \cochord(G) - 1$, for all $q \geq 2$. In \cite{jayanthan}, the
authors raised the question whether there exists a graph $G$ with
\[2q + \nu(G) - 1 < \reg(I(G)^q) < 2q + \cochord(G) - 1, \text{ for $q \gg 0$.}\] As a
consequence of our investigation, we obtain a class of graphs for which
\[2q + \nu(G) - 1 < \reg(I(G)^q) = 2q + \zeta(G) - 1 < 2q +
  \cochord(G) - 1, \text{ for $q \gg 0$.}\] 

We then proceed to prove the Conjecture \ref{ABBH-conj} for  vertex decomposable
graphs.
A graph $G$ is said to be vertex decomposable if $\Delta(G)$ is vertex decomposable, 
where $\Delta(G)$ denotes the independence complex of $G$ 
(see Section \ref{reg-vertex} for the definition).
Vertex decomposability of simplicial complexes was first introduced by Provan and Billera \cite{ProvLouis},
in the case when all the maximal faces are of the same cardinality, and
extended to the arbitrary case by Bj\"orner and Wachs \cite{BjWachs}.
We have the chain of implications
\begin{equation*}
 \text{vertex decomposable}\Longrightarrow \text{shellable} \Longrightarrow \text{sequentially
 Cohen-Macaulay.}
\end{equation*} 
A graph $G$ is said to be  shellable if $\Delta(G)$ is a shellable
simplicial complex and $G$ is sequentially Cohen-Macaulay if $R/I(G)$
is sequentially Cohen-Macaulay.
Both the above implications are known to be strict.
Recently, a number of authors have
been interested in classifying or identifying vertex decomposable graphs $G$ in
terms of the combinatorial properties of $G$, see \cite{BFH15, BC, khosh_moradi, adam, Wood2, Russ11}. 
\vskip 2mm \noindent
\textbf{Theorem \ref{scm-bipartite}.}
{\em If $G$ is a vertex decomposable graph, then for all $q \geq 1$,
} 
\[
 \reg(I(G)^q) \leq 2q+\reg(I(G))-2.
\]
\vskip 2mm 
Banerjee, Beyarslan and H\`a gave a question whether the equality $\reg(I(G)^q) = 2q +
\nu(G) - 1$ hold for all $q \geq 1$ for several important classes
of graphs, \cite[Question 7.9]{BBH17}. We answer this question
affirmatively in Corollary \ref{main-cor}.

%
%

Our paper is organized as follows. In Section \ref{pre}, we collect the terminology
and preliminary results that are essential for the rest of the paper. 
We prove, in Section \ref{tech}, several technical lemmas which are
needed for the proof of our main results which appear in Sections \ref{reg-pow} and
\ref{reg-vertex}.
%
%
%

\section{Notation and preliminaries}\label{pre}

Throughout this article, $G$ denotes a finite simple graph without isolated vertices. 
For a graph $G$, let $V(G)$ and $E(G)$ denote the set of all
vertices and the set of all edges of $G$, respectively.
The \textit{degree} of a vertex $x \in V(G),$ denoted by 
$\deg_G(x),$ is the number of edges incident to $x.$
A subgraph $H \subseteq G$  is called \textit{induced} if for $u, v
\in V(H)$, $\{u,v\} \in
E(H)$ if and only if $\{u,v\} \in E(G)$.
For $\{u_1,\ldots,u_r\}  \subseteq V(G)$, let $N_G(u_1,\ldots,u_r) = \{v \in V (G)\mid \{u_i, v\} \in E(G)~ 
\text{for some $1 \leq i \leq r$}\}$ be the set of neighbors of $u_1,\ldots,u_r$
and $N_G[u_1,\ldots,u_r]= N_G(u_1,\ldots,u_r) \cup \{u_1,\ldots,u_r\}$. 
For $U \subseteq V(G)$, we denote by $G \setminus U$ 
the induced subgraph of $G$ on the vertex set $V(G) \setminus U$.
Let $C_n$ denote the cycle on $n$ vertices.

A subset $X$ of $V(G)$ is called \textit{independent} if there is no edge $\{x,y\} \in E(G)$ 
for $x, y \in X$. 
A \textit{matching} in a graph $G$ is a subgraph consisting of pairwise disjoint edges. 
If a collection of pairwise disjoint edges is an induced subgraph,
then the matching is said to be an \textit{induced matching}. The largest size of an induced matching in $G$ is called its induced 
matching number and denoted by $\nu(G)$. 

One important tool in the study of regularity of powers of edge ideals
is even-connections.
We recall the concept of \textit{even-connectedness} from \cite{banerjee}. 

\begin{definition}\label{even_connected} Let $G$ be a graph. Two vertices $u$ and $v$ ($u$ may be the same as $v$) are said to be even-connected with respect to an $s$-fold 
products $e_1\cdots e_s$, where $e_1,\ldots,e_s$ are edges of $G$, not necessarily distinct,
if there is a path $p_0p_1\cdots p_{2k+1}$, $k\geq 1$ in $G$ such that:
\begin{enumerate}
 \item $p_0=u,p_{2k+1}=v.$
 \item For all $0 \leq l \leq k-1,$ $p_{2l+1}p_{2l+2}=e_i$ for some $i$.
 \item For all $i$, $ \mid\{l \geq 0 \mid p_{2l+1}p_{2l+2}=e_i \}\mid
   ~ \leq  ~ \mid \{j \mid e_j=e_i\} \mid$.
 \item For all $0 \leq r \leq 2k$, $p_rp_{r+1}$ is an edge in $G$.
\end{enumerate}
\end{definition}

\begin{remark} 
While we understand that the
definition of even-connection requires $k \geq 1$, for convenience
of writing the proofs, we consider an edge to be trivially
even-connected, i.e., we take the even-connection by setting $k = 0$
in the above definition.
\end{remark}
The following theorem due to Banerjee is used repeatedly throughout this paper:
\begin{theorem}\label{even_connec_equivalent}\cite[Theorem 6.1 and Theorem 6.7]{banerjee} Let $G$ be a graph with edge ideal
$I = I(G)$, and let $s \geq 1$ be an integer. Let $M$ be a minimal generator of $I^s$.
Then $(I^{s+1} : M)$ is minimally generated by monomials of degree 2, and $uv$ ($u$ and $v$ may
be the same) is a minimal generator of $(I^{s+1} : M )$ if and only if either $\{u, v\} \in E(G) $ or $u$ and $v$ are even-connected with respect to $M$.
 \end{theorem}

Polarization is a process to obtain a squarefree monomial ideal from a given monomial ideal. 
\begin{definition}\label{pol_def}
Let $M=x_1^{a_1}\cdots x_n^{a_n}$ be a monomial in  $R=\K[x_1,\dots,x_n]$. 
Then we define the squarefree monomial $P(M)$ ({\it polarization} of $M$) as 
$$P(M)=x_{11}\cdots x_{1a_1}x_{21}\cdots x_{2a_2}\cdots x_{n1}\cdots x_{na_n}$$ in the polynomial ring 
$R_1=\K[x_{ij} \mid 1\leq i\leq n,1\leq j\leq a_i]$.
If $I=(M_1,\dots,M_q)$ is an ideal in $R$, then the polarization of
$I$, denoted by $\widetilde{I}$, is defined as $\widetilde{I}=(P(M_1),\dots,P(M_q))$.
\end{definition}

Let $G$ be a graph on the vertex set $V(G) = \{x_1, \ldots, x_n\}$ and $I(G) \subset R=\K[x_1, \ldots, x_n]$ denote the
edge ideal of $G$.  For an edge $e = \{x_i, x_j\}$, we consider $e =
x_ix_j$ as an element of the polynomial ring $R$.  Let $M$ be a minimal
monomial generator of $I(G)^s$. Then $M$ can be written as product of
$s$ edges, i.e., $M=e_1 \cdots e_s$, for some edges $e_1, \ldots, e_s$,
not necessarily distinct.
By Theorem \ref{even_connec_equivalent}, $J = (I(G)^{s+1} :
e_1\cdots e_s)$ is a quadratic monomial ideal. If $x_ix_j$ is a minimal
generator of $J$, then $x_i$ and $x_j$ correspond to the vertices
of $G$ which are even-connected with respect to $e_1\cdots e_s$. If 
$x_i = x_j$, then $x_i^2$ is a minimal generator of $J$. We consider
$\tilde{J}$, the polarization of $J$, contained in the ring $R_1 =
\K[x_1, \ldots, x_n, z_1, \ldots, z_n]$ such that if $x_i^2 \in J$, then
$P(x_i^2) = x_iz_i \in \tilde{J}$. By considering $V' = \{x_1, \ldots, x_n, z_1,
\ldots, z_n\}$ as vertices, one can see that $\tilde{J}$ corresponds to
a graph, $G'$, on a vertex set $V(G') \subseteq V'$. First note that
$I(G) \subset J \subset \tilde{J}$, by considering all these ideals in
$R_1$. Consequently, we can consider $G$ as a subgraph of $G'$.
Note that $G$ may not be an induced subgraph of $G'$.

For example, let $G = C_5$ and $I(G)=(x_1x_2,x_2x_3, x_3x_4,x_4x_5,x_5x_1) \subset
\K[x_1,\ldots,x_5]$. Let $M=x_2x_3x_4x_5$ be a minimal monomial generator
of $I(G)^2$.
Then $(I(G)^3 : M) = I(G) + (x_1^2, x_1x_3,x_1x_4)$. 
Therefore,
$\widetilde{(I(G)^3 : M)} \subset \K[x_1,\ldots,x_5, z_1,\ldots z_5]$ is given by 
$\widetilde{(I(G)^3 : M)} =
I(G) + (x_1z_1,x_1x_3,x_1x_4)$.
Let $G'$ be the graph associated to 
$\widetilde{(I(G)^3 : M)}$. 
Then $V(G')=V(G) \cup \{z_1\}$ and $E(G')=E(G) \cup \{\{x_1,z_1\},\{x_1,x_3\}, \{x_1,x_4\}\}$.
Note also that $N_{G}(x_1)=\{x_2,x_5\}$ and 
$N_{G'}(x_1)=\{x_2,x_3,x_4, x_5,z_1\}$. 
 
For details of polarization we refer the reader to \cite{Herzog'sBook}.
In this paper, we repeatedly use one of its important properties, namely:
\begin{corollary} \cite[Corollary 1.6.3(a)]{Herzog'sBook}\label{pol_reg} Let $I$ be a monomial 
ideal in $\K[x_1, \ldots, x_n].$ Then
$\reg(I)=\reg(\widetilde{I}).$
\end{corollary}

\section{Technical lemmas}\label{tech}

In this section, we prove several technical results concerning the
graph associated with $(\widetilde{I(G)^{s+1} : e_1 \cdots e_s})$ and
some of its induced subgraphs. We begin by fixing the notation for the
most of our results. 

\begin{notation}\label{esetup}
Let $G$ be a graph with $V(G)=\{x_1,\ldots,x_n\}$ and $e_1,\ldots,e_s$, $s \geq 1,$ be some edges of $G$ which are not
necessarily distinct. 
By Theorem \ref{even_connec_equivalent}, $\widetilde{(I(G)^{s+1}:e_1 \cdots e_s)}$ is a
quadratic squarefree monomial ideal in an appropriate polynomial ring. We denote by 
$G'$  the graph associated to $\widetilde{(I(G)^{s+1}:e_1 \cdots e_s)}$.
\end{notation}

One of the key ingredients in the proof of the main results is a new
graph, $G'$, obtained from a given graph $G$ as in Notation \ref{esetup}.
Our main aim in this section is to get an upper
bound for regularity of certain induced subgraphs of $G'$ which in
turn will help us in bounding $\reg(I(G'))$. For this purpose, we need
to understand the structure of the graph $G'$ in more detail. First we
show that whiskers can be ignored when taking even-connections.
\begin{lemma}\label{even_obs2}
Let $G$ be a graph and $e_1,\ldots,e_s \in E(G)$ where $s \geq 1$.
Assume that for some $1 \leq i \leq s$, $e_i=\{x,y\}$ with 
$N_G(x)=\{y\}$. Then
 $$(I(G)^{s+1}:e_1 \cdots e_s)=(I(G)^s:\prod_{j\neq i}e_j).$$
\end{lemma}
\begin{proof}
By \cite[Lemma 2.10]{Morey}, we have $(I(G)^{s+1} : xy) = I(G)^s$ for
all $s \geq 1$. Hence $(I(G)^{s+1} : e_1\cdots e_s) = ((I(G)^{s+1} :
e_i) : \prod\limits_{j\neq i}e_j) = (I(G)^s : \prod\limits_{j\neq i}e_j)$.
\end{proof}

The following result shows that if a vertex has no intersection with a
set of edges, then removing such a vertex and taking even-connection
with respect to the set of those edges commute with each other.

\begin{lemma}\label{single_vertex} 
We use the notation in Notation \ref{esetup}. 
Let $G$ be a graph and $e_1,\ldots,e_s \in E(G)$ where $s \geq 1$.
If  $x \in V(G)$ satisfy $\{x\} \cap e_i=\emptyset$ for all $1 \leq i \leq s$, then 
$$I(G' \setminus x)=(\widetilde{I(G \setminus x)^{s+1}:e_1 \cdots e_s}).$$ 
\end{lemma}
\begin{proof}
Clearly $(\widetilde{I(G \setminus x)^{s+1}:e_1 \cdots e_s}) \subseteq I(G'
\setminus x)$.  
Let $u, v \in V(G \setminus x)$, not necessarily distinct, be such that $u$
is even-connected to $v$ in $G$ with respect to $e_1 \cdots e_s$.
Let $(u=p_0)p_1 \cdots p_{2k}(p_{2k+1}=v)$ be an even-connection
in $G$.  Since $e_i \cap
\{x\}=\emptyset$ for all $1 \leq i \leq s$, $u$ is even-connected to
$v$ in $G \setminus x$ with respect to $e_1 \cdots e_s$. 
\end{proof}

The next two results which throws more light into the structure of
$G'$ have been proved in \cite{JS19}. While in \cite{JS19}, the hypothesis was that the
graph is very well-covered, it may be noted that these two proofs do
not require the hypothesis. We simply recall them here without proofs.

\begin{lemma}\label{even_obs}\cite[Lemma 4.3]{JS19} 
We use the notation in Notation \ref{esetup}.
Suppose $(u=p_0)p_1 \cdots p_{2k}(p_{2k+1}=v)$ is an even-connection in $G$ with respect to
$e_1 \cdots e_s$ for some $k \geq 1$. If $\{w,p_i\} \in E(G')$ for some 
$0 \leq i \leq 2k+1$, then either $\{u,w\} \in E(G')$ or $\{v,w\} \in E(G')$.
 \end{lemma}

 \begin{lemma}\label{tech_lemma}\cite[Lemma 4.5]{JS19} Let the notation be as in Notation \ref{esetup}.
Let $y \in V(G)$ and $H=G \setminus N_G[y]$. 
If $\{e_1,\ldots,e_s\} \cap E(H) = \{e_{i_1},\ldots,e_{i_t}\}$ and $H'$ is the 
graph associated to $\widetilde{(I(H)^{t+1}:e_{i_1} \cdots e_{i_t})}$, then 
$G' \setminus N_{G'}[y]$ is an induced subgraph of $H'$.
In particular, $$\reg(I(G' \setminus N_{G'}[y])) \leq \reg(I(H')).$$
\end{lemma}

%

In the following results, we show that the even-connections in a
parent graph with respect to edges coming from an induced subgraph,
induces an even-connection in the induced subgraph.
\begin{lemma}\label{ind-reg}
Let $G$ be a graph and $H$ be an induced subgraph of $G$. For $e_1,\ldots,e_s \in E(H)$,
$s \geq 1$, let $H'$ and $G'$ be the graphs associated to $\widetilde{(I(H)^{s+1}:e_1 \cdots e_s)}$
and $\widetilde{(I(G)^{s+1}:e_1 \cdots e_s)}$ respectively. 
Then $H'$ is an induced subgraph of $G'$.
In particular, $$\reg(I(H')) \leq \reg(I(G')).$$
\end{lemma}
\begin{proof}
Let $a, b \in V(H)$, not necessarily distinct, be such that $a$ is
even-connected to $b$ in $H$ with respect to $e_1 \cdots e_s$.
For some $k \geq 0$, let $(a=p_0)p_1 \cdots p_{2k}(p_{2k+1}=b)$ be an
even-connection in $H$. Since
$H$ is an induced subgraph of $G$,
$(a=p_0)p_1 \cdots p_{2k}(p_{2k+1}=b)$ is an even-connection in $G$ with respect to 
$e_1 \cdots e_s$. Therefore
$a$ is even-connected to $b$ in $G$ with respect to $e_1 \cdots e_s$.
Hence $H'$ is a subgraph of $G'$.
Since $H$ is an induced subgraph of $G$ and $e_1, \ldots, e_s \in
E(H)$, any even-connection between vertices of $V(H)$ in $G$ with respect to $e_1\cdots e_s$ is
an even-connection in $H$ as well. Hence $H'$ is an induced subgraph
of $G'$.  The assertion on the regularity follows from
\cite[Proposition 4.1.1]{sean_thesis}.  
\end{proof}

Let the notation be as in  Notation \ref{esetup}. For some $1 \leq
\alpha \leq s$, set $e_\alpha =
\{x,y\}$. We further explore the even-connections between $N_{G'}[y]$ and
$N_{G'}(x)$. If $(u=p_0)p_1\cdots p_{2k}(p_{2k+1}=y)$ ($u$ may be
equal to $y$) is an even-connection in $G$ with respect 
to  $e_1 \cdots e_s$, then there are four possibilities:
\begin{enumerate}
	\item[(i)] $\{u,y\} \in E(G)$ i.e., $k=0$;
	\item[(ii)] $\{p_{2\lambda+1}, p_{2\lambda+2}\} \neq e_\alpha$ for any $0 \leq
	\lambda \leq k-1$;
\item[(iii)] There exists $0 \leq \lambda \leq k-1$ with $\{p_{2\lambda+1},
	p_{2\lambda+2}\} = e_\alpha$ and  $p_{2\lambda+1} = y, p_{2\lambda+2} =
	x$;
\item[(iv)] There exists $0 \leq \lambda \leq k-1$ with
	$\{p_{2\lambda+1},p_{2\lambda+2}\}=e_\alpha$, and
  for all such $\lambda$, we have $p_{2\lambda+1}=x$ and $p_{2\lambda+2}=y$.
\end{enumerate}
Let 
\begin{equation}\label{nota_nbd}
\begin{split}
	\X_y= \Big \{u \in V(G) \mid & \text{ there exists an even-connection } (u=p_0)p_1\cdots p_{2k}(p_{2k+1}=y)
 \\ & \text{ satisfies (i), (ii) or (iii) }\Big\}.
 \end{split}
\end{equation}
Note that $N_G(y) \subseteq \X_y$. It may also be noted that if $u \in
N_{G'}(y) \setminus \X_y$, then for any even-connection
$(u=p_0)p_1\cdots p_{2k}(p_{2k+1}=y)$, the conditions (i), (ii) and
(iii) are not satisfied.

We illustrate the definition of $\X_y$ with an example below. 
Let $G$ be the graph as shown in the figure below.
Let $e_1=\{x,y\}$, $e_2=\{x_3,x_4\}$, $e_3=\{x_5,x_6\}$,
$e_4=\{x_7,x_8\}$ and let $G'$ be the 
graph associated to $\widetilde{(I(G)^5:e_1e_2e_3e_4)}$.

\noindent
 \begin{minipage}{\linewidth}
\begin{minipage}{0.25\linewidth} 
\begin{figure}[H]
 \begin{tikzpicture}[scale=.55]
\draw (5,7)-- (7,7);
\draw (7,7)-- (8,8);
\draw (7,7)-- (8,6);
\draw (8,8)-- (8,6);
\draw (7,7)-- (8,9);
\draw (8,9)-- (8,10);
\draw (8,10)-- (7,10);
\draw (7,7)-- (7,9);
\draw (7,9)-- (5,9);
\draw (4,8)-- (5,7);
\draw (4,6)-- (5,7);
\draw (5,9)-- (4,8);
\begin{scriptsize}
\fill [color=black] (5,7) circle (1.5pt);
\draw[color=black] (5.14,6.66) node {$x$};
\fill [color=black] (7,7) circle (1.5pt);
\draw[color=black] (6.84,6.72) node {$y$};
\draw[color=black] (6.22,6.72) node {$e_1$};
\fill [color=black] (8,8) circle (1.5pt);
\draw[color=black] (8.28,8.26) node {$x_6$};
\fill [color=black] (8,6) circle (1.5pt);
\draw[color=black] (8.38,6.26) node {$x_5$};
\draw[color=black] (8.4,7.14) node {$e_3$};
\fill [color=black] (8,9) circle (1.5pt);
\draw[color=black] (8.48,9.06) node {$x_7$};
\fill [color=black] (8,10) circle (1.5pt);
\draw[color=black] (8.28,10.26) node {$x_8$};
\draw[color=black] (8.54,9.64) node {$e_4$};
\fill [color=black] (7,10) circle (1.5pt);
\draw[color=black] (7.28,10.26) node {$x_9$};
\fill [color=black] (7,9) circle (1.5pt);
\draw[color=black] (7.28,9.26) node {$x_4$};
\fill [color=black] (5,9) circle (1.5pt);
\draw[color=black] (4.94,9.38) node {$x_3$};
\draw[color=black] (6.22,9.46) node {$e_2$};
\fill [color=black] (4,8) circle (1.5pt);
\draw[color=black] (3.4,8.06) node {$x_2$};
\fill [color=black] (4,6) circle (1.5pt);
\draw[color=black] (3.74,6.32) node {$x_1$};
\end{scriptsize}
\end{tikzpicture}

\end{figure}
\end{minipage}
\begin{minipage}{0.74\linewidth}
Then $x_{9}x_8x_7y$ and $yx_6x_5y$  are  
even-connections in $G$ with respect to $e_1e_2e_3e_4$.
Both even-connections satisfy (ii). Hence $x_9, y
\in \X_y$. The even-connection, $x_2xyx_6x_5y$, with respect to
$e_1e_2e_3e_4$ does not satisfy, (i), (ii) and (iii). At the same time, 
$x_2x_3x_4y$ is an even-connection with respect to $e_1e_2e_3e_4$ and it
satisfies (ii). Therefore $x_2 \in \X_y$. Hence $\X_y=\{x,x_4,x_6,x_5,x_7,x_9,y,x_2\}$.
It can also be noted that $x_1 \notin \X_y$.
\end{minipage}
\end{minipage}

The following lemma will play a crucial role in the study of the regularity of
powers of edge ideals in the next section.

\begin{lemma}\label{even-lemma}
Let the notation be as  above. 
For some $1 \leq \alpha \leq s$, set $e_{\alpha}=\{x,y\}$.
 \begin{enumerate}
   \item If $N_{G'}(y) \setminus \X_y \neq \emptyset$, then $y \in \X_y$. 
   \item If $y$ is even-connected to itself, then $y \in \X_y$.
 \item If $u \in \X_y$, then $G' \setminus N_{G'}[u]$ is an
   induced subgraph of  $(G \setminus N_{G}[u,x])',$ where 
   $(G \setminus N_G[u,x])'$ is the graph
  associated to $\widetilde{(I(G \setminus N_G[u,x])^{t+1}:e_{j_1} \cdots e_{j_t})}$
   and $\{e_{j_1},
  \ldots, e_{j_t}\}=E(G\setminus N_G[u,x]) \cap \{e_1, \ldots, e_s\}$. 
   In particular,
   $$\reg(I(G' \setminus N_{G'}[u]))\leq \reg(I((G \setminus N_{G}[u,x])')).$$
  \item The graph $G' \setminus \X_y$ is an induced subgraph of $(G \setminus N_G[y])'$, where 
  $(G \setminus N_G[y])'$ is the graph
  associated to $\widetilde{(I(G \setminus N_G[y])^{t+1}:e_{j_1} \cdots e_{j_t})}$
   and $\{e_{j_1},
  \ldots, e_{j_t}\}=E(G\setminus N_G[y]) \cap \{e_1, \ldots, e_s\}$. 
   In particular,
   $$\reg(I(G' \setminus \X_y))\leq \reg(I((G \setminus N_{G}[y])')).$$
 \end{enumerate} 
\end{lemma}
\begin{proof}
(1) Let $u\in N_{G'}(y) \setminus \X_y$ and $(u=p_0)p_1 \cdots (p_{2k+1}=y)$ be an even-connection in $G$
with respect to $e_1 \cdots e_s$.
Since $u \notin \X_y$, there exists $0 \leq \lambda \leq k-1$ with 
$\{p_{2\lambda+1},p_{2\lambda+2}\}=e_\alpha$ with $p_{2\lambda +1} =
x$ and $p_{2\lambda+2} = y$. Let $\gamma$ be the
largest integer with this 
property. Note that $p_{2\gamma+2}=y$.
Then $(y=p_{2\gamma+2})p_{2\gamma+3} 
 \cdots p_{2k}(p_{2k+1}=y)$ is an even-connection in $G$ and
 $\{p_{2\lambda'+1},p_{2\lambda'+2}\} \neq e_\alpha$
 for all $\gamma+1 \leq \lambda' \leq k-1$.  Therefore
 $y \in \X_y$. 
\vskip 1mm \noindent
(2) 
Assume, in contrary, that
any even-connection from $y$ to $y$ with respect
to $e_1 \cdots e_s$
 does not satisfy (ii) and (iii). Then by taking $u =
y$ in the proof of (1), it can be seen that there exists an
even-connection from $y$ to itself which satisfy the condition (ii).
Hence $y \in \X_y$.

\vskip 1mm \noindent
(3) Set $H=G' \setminus N_{G'}[u]$ and $K=(G \setminus N_{G}[u,x])'$. 
Let $a, b \in V(H)$, not necessarily distinct, be such that $a$ is
even-connected to $b$ in $G$ with respect to $e_1 \cdots e_s$. 
Let $(a=q_0)q_1 \cdots (q_{2l+1}=b)$ be an
even-connection in $G$ with respect to $e_1 \cdots e_s$.
We claim that $q_r \notin N_G[u,x]$ for all $0 \leq r \leq 2l+1$.
Note that by Lemma \ref{even_obs}, $q_r \notin N_{G'}[u]$ for all $0 \leq r \leq 2l+1$.
Suppose
$q_r \in N_{G}[x]$ for some $0 \leq r \leq 2l+1$. Since $u \in \X_y$, $u$ is even-connected to 
$q_r$ in $G$ with respect to $e_1 \cdots e_s$.
By Lemma \ref{even_obs}, $u$ is even-connected either to
$a$ or to $b$ in $G$ with respect to $e_1 \cdots e_s$.  Hence either
$a$ or $b$ belongs to $N_{G'}[u]$.
This is a contradiction to our assumption that
$a,b \in V(H)$.
Hence $q_r \notin N_{G}[x]$ for all $0 \leq r \leq 2l+1$. Therefore
$a$ is even-connected to $b$ in $G \setminus N_G[u,x]$ with respect to
$e_{j_1} \cdots e_{j_t}$. 
\vskip 1mm \noindent
(4) Let  $a, b \in V(G' \setminus \X_y)$ and 
$a$ be even-connected to $b$ in $G$ with respect to 
$e_1 \cdots e_s$.
Using the fact that $N_G(y) \subseteq \X_y$ and by (1) and (2), we get
$a,b \notin N_G[y]$. Hence, if $\{a, b\} \in E(G)$, then $\{a,b\} \in
E(G \setminus N_G[y])$.
 Let
$(a=q_0)q_1 \cdots (q_{2l+1}=b)$ be an
even-connection in $G$ with respect to $e_1 \cdots e_s$. 
We claim that $q_j \notin N_G[y]$ for any $j$. Suppose $q_j \in N_G(y)$ for some 
$0 \leq j \leq 2l+1$. If $j$ is odd, then choose the largest integer $r$ such that $q_r \in N_G(y)$.
Then $(b=q_{2l+1})q_{2l} \cdots q_r y$ is an even-connection in $G$ with respect to 
$e_1 \cdots e_s$. Since $r$ is the largest integer, $q_j \notin N_G[y]$ for all $r < j \leq 2l+1$.
Therefore $b \in \X_y$ which is a contradiction. Now if $r$ is even,
then choose the smallest integer $r$ such that $q_r \in N_G(y)$.
Then $(a=q_{0})q_{1} \cdots q_r y$ is an even-connection in $G$ with respect to 
$e_1 \cdots e_s$. Therefore $a \in \X_y$ which again is a
contradiction. Hence $q_j \notin N_G(y)$ for all $j$.
If $q_j=y$ for some $0 \leq j \leq 2l+1$, then $1\leq j \leq 2l$ either $q_{j-1} \in N_G(y)$ or 
$q_{j+1} \in N_G(y)$ which contradicts the first part of the proof.
This completes the proof of the claim.
This shows that $a$ is even-connected to $b$ in $G \setminus N_G[y]$
with respect to $e_{j_1} \cdots e_{j_t}$.

As in Lemma \ref{ind-reg}, it can be seen that the subgraphs
considered in (3) and (4) are induced subgraphs.
The assertion on the regularity in (3) and (4) follows from \cite[Proposition
4.1.1]{sean_thesis}.  
\end{proof}

\section{Regularity of powers of graphs}\label{reg-pow}

In this section, we obtain a general upper bound for the regularity of
powers of edge ideals of graphs. The first main theorem gives certain
sufficient conditions for any combinatorial invariant to be an upper
bound for the constant term of the linear polynomial corresponding to
$\reg(I(G)^q)$. The below result can be seen as a different version of
\cite[Theorem 3.3]{BBH19}.

\begin{theorem}\label{main-result}
Let $G$  be a graph 
and $\rho: \mathcal{I}_G \longrightarrow 
\mathbb{N}$ be a function such that for any $L \in \mathcal{I}_G$,
\begin{enumerate}
 \item $\reg(I(L)) \leq \rho(L)+1$,
 \item $\rho(L_1) \leq \rho(L)$ for any induced subgraph $L_1$ of $L$
   and
 \item there exists a vertex $x \in V(L)$ such that $\rho(L \setminus N_L[x])+1 \leq \rho(L)$.
\end{enumerate}
Then 
\[
 \reg(I(G)^q) \leq 2q+\rho(G)-1 ~\text{ for all $q \geq 1$.}
\]
\end{theorem}
\begin{proof}
Let $G$ be a graph and $\rho : \mathcal{I}_G \longrightarrow
\mathbb{N}$ be a function satisfying the given hypotheses. We prove the
assertion by induction on $q$.
The case $q=1$ follows from the assumption.
Assume that $q >1$. 
For any graph $K$, set 
 \[
  \PP(K)=\{x \in V(K) \mid \rho(K) \geq \rho(K \setminus N_K[x])+1\}.
   \]
By hypothesis, $\PP(G) \neq \emptyset$.  By applying \cite[Theorem
5.2]{banerjee} and using induction, it is enough to prove that for
edges $e_1, \ldots, e_s$ of $G$, $\reg((I(G)^{s+1}:e_1 \cdots e_s))
\leq \rho(G)+1$ for all $s \geq 0$.  Let $G'$ be the graph associated
to the ideal $\widetilde{(I(G)^{s+1}:e_1\cdots e_s)}$ which is
contained in an appropriate polynomial ring $R_1$. We prove that
$\reg(I(G')) \leq \rho(G) + 1$ by induction on $s+|V(G)|$.  

If $s = 0$, then $G' = G$ and hence the assertion is true for any
value of $|V(G)|$. Therefore, we may assume that $s \geq 1$. If
$|V(G)| = 2$, then $G$ consists of only one edge. In this case, we also have $G' = G$ and hence the
assertion is true. Now, assume that $s \geq 1$ and $|V(G)| > 2$.

Let $e_i = \{a_i, b_i\}$ for $1 \leq i \leq s$. If
$\deg_G(a_i) = 1$ or $\deg_G(b_i) = 1$ for some $i$, then by Lemma
\ref{even_obs2}, it follows that 
\[\reg(I(G)^{s+1} : e_1\cdots e_s) = \reg(I(G)^s : e_1\cdots
  e_{i-1}e_{i+1}\cdots e_s) \leq \rho(G)+1,\]
where the last inequality follows from the hypothesis of induction. 

Assume now that $\deg_G(a_i) \geq 2$ and $\deg_G(b_i) \geq 2$ for all
$1 \leq i \leq s$.
%
  \vskip 1mm \noindent
\textsc{Case 1:} Suppose $e_i \cap \PP(G) \neq \emptyset$ for some $1
\leq i \leq s$. 
\vskip 1mm \noindent
Without loss of generality, we may assume that $e_s \cap \PP(G) \neq
\emptyset$ and $a_s \in \PP(G)$. Let $J=I(G')$.
Following the notation as in (\ref{nota_nbd}),
set $\X_{b_s}=\{y_1,\ldots,y_p\}$.
It follows from the collection of short exact sequences:
\begin{eqnarray*}\label{main-exact-seq}
  0 & \longrightarrow & \frac{R_1}{(J : y_1)}(-1)
  \overset{\cdot y_1}{\longrightarrow} \frac{R_1}{J} \longrightarrow
  \frac{R_1}{J + (y_1)} \longrightarrow 0;\nonumber \\
& & \hspace*{0.5cm} \vdots\hspace*{3cm}\vdots \hspace*{3cm}\vdots
\\ 
0 & \longrightarrow & \frac{R_1}{( (J+
  (y_1,\ldots,y_{p-1})):y_p)}(-1) 
\overset{\cdot y_p}{\longrightarrow}
\frac{R_1}{J+(y_1,\ldots,y_{p-1})} \longrightarrow
\frac{R_1}{J+ (\X_{b_s})} \longrightarrow 0,\nonumber
\end{eqnarray*}
that
\[
  \reg(R_1/J)  \leq  \max \left\{
	\begin{array}{l}
	  \reg\left(\frac{R_1}{(J :y_1)}\right)+1, \ldots, 
		\reg\left(\frac{R_1}{((J+(y_1,\ldots,y_{p-1})):y_p)}\right) + 1,~ 
		\reg\left(\frac{R_1}{J+(\X_{b_s})}\right)
\end{array}\right.
\Big\}.
\]
 
Now,
\[
  \begin{array}{llll}
\reg(J+(\X_{b_s}))& = & \reg(I(G' \setminus \X_{b_s})) &
(\text{by \cite[Remark 2.5]{selvi_ha}}) \\
 & \leq&  \reg(I((G \setminus
 N_{G}[b_s])')), & (\text{by Lemma \ref{even-lemma}(4)})
 \end{array}\]
 where  
 $E(G \setminus N_G[b_s]) \cap
 \{e_1,\ldots,e_s\}=\{e_{j_1},\ldots,e_{j_t}\}$ and $(G \setminus N_{G}[b_s])'$
 is the graph associated to $\widetilde{(I(G \setminus
   N_G[b_s])^{t+1}:e_{j_1} \cdots e_{j_t})}$.   
Therefore
$$\reg(J+ (\X_{b_s})) \leq \reg(I( (G \setminus N_G[b_s])')) \leq \rho(G \setminus N_G[b_s])+1 \leq \rho(G)+1,$$
where the second and last inequalities follows by induction on the number of
vertices and the assumption (2) respectively.
Using similar arguments, we get
$$
\begin{array}{llll}
  \reg((J : y_i))=\reg(I(G' \setminus N_{G'}[y_i])) & \leq & \reg(I((G \setminus N_G[y_i,a_s])')) &
  (\text{by Lemma } \ref{even-lemma}(3)) \\
  & \leq & \reg(I((G \setminus N_G[a_s])')) & (\text{by Lemma }
	\ref{ind-reg}) \\
	& \leq&  \rho(G \setminus N_G[a_s])+1 & (\text{by induction}) \\
	& < & \rho(G)+1, &
\end{array}
$$
where the last inequality follows by  the assumption that  $a_s \in \PP(G)$.

Note that $( (J+(y_1,\ldots,y_{i-1})):y_i)$ is the edge ideal of 
$(G' \setminus N_{G'}[y_i]) \setminus \{y_1,\ldots,y_{i-1}\}$ and 
$(J:y_i)$ is the edge ideal of $G' \setminus N_{G'}[y_i]$. 
Since $(G' \setminus N_{G'}[y_i]) \setminus \{y_1,\ldots,y_{i-1}\}$ is an 
induced subgraph of $G' \setminus N_{G'}[y_i]$, it follows that
$$\reg(( (J+(y_1,\ldots,y_{i-1})):y_i)) \leq \reg(J:y_i)< \rho(G)+1.$$

Therefore  $\reg(J) \leq \rho(G)+1$.
\vskip 1mm \noindent
\textsc{Case 2:}
Suppose $e_i \cap \PP(G) = \emptyset$ for all $1 \leq i \leq s$.
Let $x \in \PP(G)$. Then by \cite[Theorem 3.4]{Ha2}, 
$$\reg(I(G')) \leq \max
\Big\{\reg(I(G' \setminus x)), \reg(I(G' \setminus
N_{G'}[x]))+1\Big\}.$$
By Lemma \ref{single_vertex} and inductive hypothesis we get 
\[\reg(I(G' \setminus x)) =
\reg(I(G\setminus x)^{s+1} : e_1 \cdots e_s)) \leq \rho(G \setminus x)
+ 1 \leq \rho(G) + 1.\] 
Similarly, by Lemma \ref{tech_lemma} and inductive hypothesis we get  
\[\reg(I(G' \setminus N_{G'}[x])) \leq \reg(I(G
\setminus N_G[x])^{t+1} : e_{i_1}\cdots e_{i_t}) \leq \rho(G \setminus
N_G[x]) + 1 \leq \rho(G),\] where 
the last inequality follows by the assumption that $x \in \PP(G)$, and 
$ \{e_{i_1}, \ldots, e_{i_t}\}=E(G\setminus
N_G[x]) \cap \{e_1, \ldots, e_s\}$.
Therefore $\reg(I(G')) \leq \rho(G)+1$. 

This
completes the proof.
\end{proof}

As a consequence of Theorem \ref{main-result}, we obtain a sufficient
condition for the Conjecture \ref{ABBH-conj} to be true.
\begin{corollary} \label{weak-conj1} Let $G$ be a graph. If every non-empty
induced  subgraph $H$ of 
$G$ has a vertex $x$ with $\reg(I(H \setminus N_H[x]))+1 \leq
\reg(I(H))$, then for all $q \geq 1$,
\[
 \reg(I(G)^q) \leq 2q+\reg(I(G))-2.
\] 
\end{corollary}
\begin{proof}
For $L \in \mathcal{I}_G$, let $\rho(L) = \reg(I(L)) - 1$. 
Then it is
easy to see that $\rho$ satisfies (1) - (3) of Theorem
\ref{main-result}. Hence the assertion follows.
\end{proof}

It is interesting to ask if every graph $G$ has a vertex $x$ with
$\reg(I(G \setminus N_G[x])) + 1 \leq \reg(I(G))$.
It was communicated to us by Tran Nam Trung that there exists a graph
$G$ which does not satisfy the hypothesis of Corollary \ref{weak-conj1}.
Let $G$ be the graph denoted by $G_2$ in Appendix A of \cite{kat},
page 452. Then for every vertex $x$ of $G$, it can be verified that
$\reg(I(G)) = \reg(I(G \setminus N_G[x]))$.
Hence we would like to ask:

\begin{question}
Can we classify graphs $G$ having
a vertex $x$ such that $\reg(I(G \setminus N_G[x])) + 1 \leq \reg(I(G))$?
\end{question}

As more applications of Theorem \ref{main-result}, we obtain upper
bounds for the regularity of powers of edge ideals.
We first recall the
definitions of the invariants $\cochord(G)$ and $\zeta(G)$. 

The \textit{complement} of a graph $G$, denoted by $G^c$, is the graph on the same
vertex set as $G$ in which $\{u,v\}$ is an edge of $G^c$ if and only if it is not an edge of $G$.
A graph $G$ is \textit{chordal} if every induced
cycle in $G$ has length $3$, and is co-chordal if the complement graph $G^c$ is chordal.
The \textit{co-chordal cover number}, denoted $\cochord(G)$, is the
minimum number $n$ such that there exist co-chordal subgraphs
$H_1,\ldots, H_n$ of $G$ with $E(G) = \bigcup_{i=1}^n E(H_i)$.

Now we recall the definition of $\zeta(G)$ from \cite{HaWood}.
 A \textit{star} at $x$, which is the subgraph on $N_G[x]$ with edge set consisting of all edges 
 of $G$ incident to $x$. We say that a star is \textit{nondegenerate} if $\deg_G(x) > 1$, 
 so that the star does not consist of a single vertex or a single edge. 
 We say a set
of stars is \textit{center-separated} if the center of a star and at least two of its neighbors are not 
contained in any other star. 
In the collection of all stars in $G$, 
a \textit{maximal center-separated} star packing of $G$ is a
center-separated star packing of $G$ that is not a subset of any other
center-separated star packing.
Let $\PP$ be a maximal center-separated star packing.
After deleting the vertices of the stars  $\PP$, 
an induced matching of $G$ will remain. Let $\zeta_{\PP}$ be the number of stars
in the packing plus the number of edges in 
the remained induced matching and let $\zeta(G)$ be the maximum $\zeta_{\PP}$ over all maximal center-separated packings of 
nondegenerate stars. 

For example, if $G = C_n$, cycle on $\{x_1,\ldots,x_n\}$ vertices, then for any $x \in
V(G)$, star at $x$ is a path on $3$ vertices and hence $G \setminus
N_G[x]$ is a path on $n-3$ vertices. If $n\equiv 0,1(mod~3)$, then 
$\PP= \{$ star at $x_1$,
star at $x_4$,  star at $x_7,\ldots, $ star at $x_{n-2}\}$
is a center-separated star packing of $G$.
Therefore $\zeta_{\PP}=\lfloor \frac{n}{3} \rfloor$.
If $n \equiv 2 (mod~3)$, then 
$\PP'=\{$ star at  $x_1$,  star at  $x_4,\ldots,$  star at $x_{n-4}\}$
is a center separated star packing of $G$. Note that $G \setminus \PP'$
consists of a single edge.
Hence $\zeta_{\PP'}=\lfloor \frac{n}{3} \rfloor +1$.
Therefore, if $n \equiv 0,1(mod~3)$, then $\zeta(G) \geq \lfloor \frac{n}{3} \rfloor$ and 
if $n \equiv 2 (mod~3)$, then $\zeta(G) \geq \lfloor \frac{n}{3} \rfloor+1$. It is not hard to 
verify that the above inequalities are in fact equalities.
Also, note that co-chordal graphs do not have two disjoint edges. 
Therefore, we have:

\begin{enumerate}
  \item if $n \equiv 0 (mod~3)$ or $n=4$, then $\nu(C_n)=\cochord(C_n)=\zeta(C_n)=\lfloor\frac{n}{3}\rfloor$;
  \item  if $n \equiv 1 (mod~3)$ and $n>4$, then $\nu(C_n)=\zeta(C_n)=\cochord(C_n)-1=\lfloor \frac{n}{3} \rfloor$;
  \item if $n \equiv 2 (mod~3)$, then $\nu(C_n)=\zeta(C_n)-1=\cochord(C_n)-1=\lfloor \frac{n}{3} \rfloor$.
 \end{enumerate}

It may be noted that for any graph $G$, $\nu(G) \leq \zeta(G)$.
H\`a and Woodroofe proved that for a graph $G$, 
$\reg(I(G)) \leq \zeta(G)+1,$ \cite{HaWood}.
Also, it was proved by Woodroofe that 
$\reg(I(G)) \leq \cochord(G) + 1$, \cite[Theorem
1]{russ}. We would like to note here that the invariants
$\zeta(G)$ and $\cochord(G)$ are not comparable in
general. 
For example, if $G=C_7$, then $\cochord(G)=3$ and $\zeta(G)=2$. If $H$ is
the graph with $E(H)=\Big\{ \{x_1,x_2\}, \{x_2,x_3\},\{x_3,x_4\},\{x_4,x_1\},\{x_1,x_5\}\Big\}$,
then it is easy to see that $\cochord(H)=1$ and $\zeta(H)=2$. Now we
prove one of the main results of this paper.

\begin{theorem}\label{mainresult}

Let $G$  be a graph.  Then for all $q \geq 1$,
\begin{enumerate}
 \item $\reg(I(G)^q) \leq 2q+\zeta(G)-1.$
 \item $\reg(I(G)^q) \leq 2q+\cochord(G)-1$.
\end{enumerate}
\end{theorem}
\begin{proof}
From \cite[Theorem 1.6]{HaWood} and \cite[Theorem 1]{russ}, it follows that for any graph $G$,
$\reg(I(G)) \leq \zeta(G)+1$ and $\reg(I(G)) \leq \cochord(G)+1$.
It is easy to see that, for any induced subgraph $L$ of $G$, $\zeta(L) \leq \zeta(G)$ and 
$\cochord(L) \leq \cochord(G)$.
\vskip 1mm \noindent
(1) 
Suppose that there exists a 
vertex $x$ in $G$ such that $\deg_G(x) \geq 2$. 
Let $H = G \setminus N_G[x]$. Let $\PP'$ be a maximal center-separated star packing of $H$ such
that $\zeta_{\PP'}(H) = \zeta(H)$. Then $\PP = \PP' \cup \{\text{ star at } x\}$
is a center-separated star packing of $G$. Thus, 
$\zeta(H) + 1 = \zeta_{\PP'}(H) + 1 = \zeta_{\PP}(G) \leq \zeta(G).$
If $\deg_G(x)=1$ for all $x \in V(G)$, then by 
definition $\zeta(G \setminus N_G[x])+1 \leq \zeta(G)$ for all $x \in V(G)$. 
Therefore, by Theorem \ref{main-result}, $\reg(I(G)^q) \leq 2q+\zeta(G)-1.$
\vskip 1mm \noindent
(2) 
It follows from 
\cite[Lemma 3.1]{SY18} that there is a vertex $x \in V(G)$ such that 
$\cochord(G \setminus N_G[x])+1 \leq \cochord(G)$. Hence, by Theorem \ref{main-result}, $\reg(I(G)^q) \leq 2q+\cochord(G)-1$.
\end{proof}
A graph $G$ is said to be \textit{weakly chordal} if
neither $G$ nor $G^c$ contain induced cycle of length $5$ or more.
It is straightforward to show that a chordal graph is weakly chordal.
The \textit{matching number} of $G$, denoted by $\ma(G)$, is the
maximum cardinality among matchings of $G$ 
and the minimum matching number of $G$, denoted by $\minmax(G)$, is
the minimum cardinality among maximal matchings of $G$. 
By \cite[p. 2]{hibi}, \cite[Theorem 1]{russ}, \cite[p. 10]{jayanthan},
for any graph $G$, we have
\[
 \nu(G) \leq \cochord(G) \leq \minmax(G) \leq \ma(G).
\]

Many authors have studied classes of graphs whose induced matching number coincides
with $\cochord(G), \minmax(G)$ or $\ma(G)$. For example, it is known
that $\nu(G) = \cochord(G)$ for 
unmixed bipartite graphs (\cite[Theorem 16]{russ}), 
weakly chordal graphs (\cite[Proposition 3]{busch_dragan_sritharan})
and bipartite graphs with $\reg(I(G))=3$ (\cite[Observation 5.3]{jayanthan}).
For
Cameron-Walker graphs $\nu(G) = \ma(G)$,  (\cite{CamWalker, HHKO}).
Hibi et al. studied the class of graphs for which $\nu(G) =
\minmax(G)$, \cite{hibi}. 
Beyarslan, H\`a and Trung proved that $\reg(I(G)^q) \geq 2q + \nu(G) -
1$ for any graph $G$ and for all $q \geq
1$, \cite[Theorem 4.5]{selvi_ha}. Hence, we can use Theorem
\ref{mainresult} for the above mentioned classes
of graphs to get $\reg(I(G)^q)=2q+\nu(G)-1$ for all $q \geq 1$.

\begin{corollary}\label{known-results} The following hold:
 \begin{enumerate}
 \item If $G$ is a weakly chordal graph, then 
  $\reg(I(G)^q)=2q+\nu(G)-1$ for all $q \geq 1$.
  \item If $G$ is a Cameron-Walker graph, then $\reg(I(G)^q)=2q+\nu(G)-1$ for all $q \geq 1$.
 \item \cite[Theorem 3.2]{HHZ} If $I(G)$ has a linear resolution (i.e., $G^c$ is chordal), then $I(G)^q$ has a linear resolution for all $q \geq 2$.
  \item \cite[Theorem 4.7]{selvi_ha} If $G$ is a forest, then $\reg(I(G)^q)=2q+\nu(G)-1$ for all $q \geq 1$.
  \item \cite[Corollary 5.1(1)]{jayanthan} If $G$ is an unmixed bipartite graph, then $\reg(I(G)^q)=2q+\nu(G)-1$ for all $q \geq 1$.
  \item \cite[Theorem 3.9]{banerjee1} If $G$ is a bipartite graph and $\reg(I(G))=3$, then 
  $\reg(I(G)^q)=2q+1$ for all $q \geq 1$.
 \end{enumerate}
\end{corollary}

For two vertex disjoint graphs $G_1$ and $G_2$, we denote the union of
$G_1$ and $G_2$ by $G_1 \coprod G_2$ i.e., $V(G_1 \coprod G_2)=V(G_1)
\cup V(G_2)$ and $E(G_1 \coprod G_2)=E(G_1) \cup E(G_2)$.  Note that
$\Psi(G_1 \coprod G_2)=\Psi(G_1)+\Psi(G_2)$ whenever $\Psi(-)$ is
equal to $\nu(-)$, $\cochord(-)$, $\minmax(-)$, $\ma(-)$ or
$\zeta(-)$.  If $G_1$ and $G_2$ are graphs for which the linear
polynomials corresponding to $\reg(I(G_1)^q)$ and $\reg(I(G_2)^q)$ are
known, then using \cite[Theorem 5.7]{nguyen_vu} it is possible to
compute the linear polynomial corresponding to $\reg(
(I(G_1)+I(G_2))^q)=\reg(I(G_1 \coprod G_2)^q)$. 
\begin{proposition}\label{disjoint-reg} 
Let $G_1=(\coprod_{i=1}^t C_{n_i})$
where $t \geq 1$ and $n_1,\ldots,n_t \equiv 2(mod~3)$.
Let $G_2$ be an arbitrary graph and set 
$G=G_1\coprod G_2$.
\begin{enumerate}
 \item Then $\reg(I(G_1)^q)=2q+\nu(G_1)+t-2$ for all $q \geq 2$.
 \item If $\reg(I(G_2)^q)=2q+\nu(G_2)-1$ for all $q \geq 1$, 
 then $\reg(I(G)^q)=2q+\nu(G)+t-1$ for all $q \geq 2$.
\end{enumerate} 
\end{proposition}
\begin{proof}
By \cite[Lemma 8]{russ}, \cite[Theorem 7.6.28]{sean_thesis},
$\reg(I(G_1))=\nu(G_1)+t+1$.
\vskip 1mm
\noindent
 (1) 
We prove  by induction on $t$. If $t=1$, then the assertion follows from 
\cite[Theorem 5.2]{selvi_ha}. 
 Set $H=\coprod_{i=1}^{t-1}C_{n_i}$. Also, from \cite[Lemma 8]{russ}, we get
$\reg(I(H)) = \nu(H) + t$.
By inductive hypothesis on $t$, 
$\reg(I(H)^q)=2q+\nu(H)+(t-1)-2 \text{ for all } q \geq 2.$
Note that $G_1=H \coprod C_{n_t}$.
  By \cite[Proposition 2.7]{HTT},
$\reg(I(G_1)^2)=2+\nu(G_1)+t.$
Using \cite[Theorem 5.2]{selvi_ha} and \cite[Theorem 1.1]{nguyen_vu}, we get
$\reg(I(G_1)^3)=4+\nu(G_1)+t.$
When $q \geq 4$, we can apply \cite[Theorem 5.7]{nguyen_vu}
by taking 
$I = I(H)$ and $J = I(C_{n_t})$.
Note that $g = \nu(H) + t-3$, $g^*
= \nu(H) + t-2$, $h = \nu(C_{n_t})-1$ and $h^* = \nu(C_{n_t})$
with notation of \cite[Theorem 5.7]{nguyen_vu}.
Then we get 
$\reg(I(G_1)^q)=2q+\nu(G_1)+t-2$ for all $q \geq 4$.

\vskip 1mm
\noindent
(2) It follows from \cite[Proposition 2.7]{HTT} that
$\reg(I(G)^2)=3+\nu(G)+t$. 
Similarly to the previous case, by applying \cite[Theorem 5.7]{nguyen_vu}
with $I=I(G_1)$ and $J=I(G_2)$, we have 
$\reg(I(G)^q)=2q+\nu(G)+t-1$ for all $q \geq 3$.
\end{proof}

Trung proved that for a graph $G$, $\reg(I(G)) = \ma(G)+1$ if and only
if  each connected component of $G$ is either $C_5$ or a
Cameron-Walker graph, \cite[Theorem 11]{Trung18}. As an immediate
consequence of our previous result, we compute the regularity of
all powers of such graphs.
\begin{proposition}\label{mat-reg}
If $\reg(I(G))=\ma(G)+1$, then either $\reg(I(G)^q)=2q+\ma(G)-2$ for
all $q \geq 2$ or $\reg(I(G)^q)=2q+\ma(G)-1$ for all $q \geq 2$.
\end{proposition}
\begin{proof}
Since $\reg(I(G))=\ma(G)+1$, by \cite[Theorem 11]{Trung18}, it follows that
$$G= (\coprod_{i=1}^t C_5)
\coprod (\coprod_{j=1}^l H_j),$$ where $H_j$ are Cameron-Walker
graphs, for some $t, l \geq 0$. Note that $\nu(\coprod_{i=1}^t
C_5)=t=\ma(\coprod_{i=1}^t C_5)-t$ and $\nu(\coprod_{j=i}^l H_j)=\ma(\coprod_{j=i}^l H_j)$.
If $l=0$, then by Proposition \ref{disjoint-reg}(1), $\reg(I(G)^q)=2q+\nu(G)+t-2$ 
for all $q \geq 2$.  Therefore
$\reg(I(G)^q)=2q+\ma(G)-2$ 
for all $q \geq 2$.
Suppose $t=0$.  Then by Theorem \ref{mainresult} and 
\cite[Theorem 4.5]{selvi_ha}, $\reg(I(G)^q)=2q+\nu(G)-1=2q+\ma(G)-1$ for all $q \geq 2$. If 
$t>0$ and $l>0$, then by Proposition \ref{disjoint-reg}(2), for all $q \geq 2$, 
$\reg(I(G)^q)=2q+\nu(G)+t-1=2q+
\nu(\coprod_{i=1}^t C_5)+\nu(\coprod_{j=1}^l H_j)+t-1=2q+\ma(G)-1.$
\end{proof}

Using Proposition \ref{disjoint-reg},
we obtain a class of graphs for which the upper bound in
Theorem \ref{mainresult}(1) is attained.

\begin{proposition}\label{disjoint}
 For $p \geq 0$ and $r > p$, let
$H=\left(\coprod_{i=1}^p C_{n_i}\right) \coprod \left(\coprod_{j=p+1}^rC_{n_j}\right),$ where 
$n_1,\ldots,n_p \equiv 2(mod~3)$ and $n_{p+1},\ldots,n_r \equiv
0,1(mod~3)$.
Then for all $q \geq 1$,
$$\reg(I(H)^q)=2q+\zeta(H)-1.$$ 
\end{proposition}
\begin{proof}
If $q=1$, then by \cite[Lemma 8]{russ} and \cite[Theorem 7.6.28]{sean_thesis}, we get $\reg(I(H))=\nu(H)+p+1=\zeta(H)+1$. 
Suppose $p=0$. Using \cite[Theorem 5.2]{selvi_ha} and
\cite[Theorem 5.7]{nguyen_vu}, we get 
$\reg(I(H)^q)=2q+\nu(H)-1=2q+\zeta(H)-1 \text{ for all $q \geq 1$}.$
When $p \neq 0$, we can apply Proposition \ref{disjoint-reg}(2) with 
$G_1= \coprod_{i=1}^p C_{n_i}$ and $G_2= \coprod_{j=p+1}^r C_{n_j}$
and get 
$\reg(I(H)^q)=2q+\nu(H)+p-1=2q+\zeta(H)-1$ for all $q \geq 2$.
\end{proof}

In \cite{jayanthan}, the authors asked if there exists a graph 
$G$ with $2q + \nu(G) - 1 < \reg(I(G)^q) < 2q + \cochord(G) - 1$ for
all $q \gg 0$, \cite[Question 5.8]{jayanthan}. We show that some of
the graphs considered in Proposition
\ref{disjoint} satisfy this inequality.
Let $H$ be a graph as in Proposition \ref{disjoint}, with $n_j \equiv
1(\text{mod }3)$ and $n_j>4$ for $j = p+1, \ldots, r$, $p>0$. Then $\nu(H) =
\sum_{i=1}^r \lfloor\frac{n_i}{3} \rfloor$, $\zeta(H) =
p+ \sum_{i=1}^r \lfloor \frac{n_i}{3} \rfloor$ and $\cochord(H) =
r+ \sum_{i=1}^r \lfloor \frac{n_i}{3} \rfloor$. Therefore, we get for all $q \geq 1$,
$$2q+\nu(H)-1<\reg(I(H)^q) = 2q + \zeta(H) - 1 <2q+\cochord(H)-1.$$

\section{Regularity of powers of vertex decomposable graphs}\label{reg-vertex}
In this section, we prove Conjecture \ref{ABBH-conj} for vertex decomposable graphs.
We first recall the definition of a simplicial complex and a vertex decomposable graph.

A \textit{simplicial complex} $\Delta$ on $V = \{x_1,\ldots,x_n\}$ is
a collection of subsets of $V$ such that: 
\begin{enumerate}
 \item $\{x_i\}\in \Delta $ for $i =1,\ldots,n$, and
 \item if $F \in \Delta$ and $G \subseteq F$, then $G \in \Delta$.
\end{enumerate}
Elements of $\Delta$ are called the \textit{faces} of $\Delta$, and the maximal elements, with 
respect to inclusion, are called the facets. The link of a face $F$ in
$\Delta$ is $\link_\Delta(F) = \{F' \in \Delta \mid F' \cup F \in \Delta, ~F' \cap F = \emptyset \}$.

A simplicial complex $\Delta$ is
recursively defined to be {\em vertex decomposable} if it is either a
simplex or else has some vertex $v$ such that 
\begin{enumerate}
  \item both $\Delta \setminus v$ and $\link_\Delta v$ are vertex decomposable, and
  \item no face of $\link_\Delta v$ is a facet of $\Delta \setminus v$.
\end{enumerate}

The \textit{independence complex} of $G$, denoted by $\Delta(G)$, is the simplicial
complex on $V(G)$ with face set 
$$\Delta(G)=\Big\{F \subseteq V(G) \mid F \text{ is an independent set of $G$ } \Big\}.$$
A graph $G$ is said to be \textit{vertex decomposable} if $\Delta(G)$ is a
vertex decomposable simplicial complex. 
In \cite{Wood2}, Woodroofe translated the notion of vertex decomposable
for graphs as follows.
\begin{definition}\cite[Lemma 4]{Wood2} \label{wood-def}
 A graph $G$ is recursively defined to be vertex decomposable if 
 $G$ is totally disconnected (with no edges) or if there is a vertex $x$ in $G$
 such that
 \begin{enumerate}
  \item   $G \setminus x$ and 
  $G \setminus N_G[x]$ are both vertex decomposable, and
  \item no independent set in $G\setminus N_G[x]$ is a maximal independent set 
  in $G\setminus x$.
 \end{enumerate}
\end{definition}
A vertex $x$ which satisfies the second condition of Definition \ref{wood-def} is called a \textit{shedding vertex} of
$G$.
 If $G$ is a vertex decomposable graph, then by \cite[Theorem 2.5]{BFH15},
 $G \setminus N_G[x]$ is a vertex decomposable graph, for any  $x \in V(G)$.
 For any vertex decomposable graph $K$, set 
 $$
\s(K)=\Big \{ x \in V(K) \mid \text{$x$ is a shedding vertex and 
$K \setminus x$ is a vertex decomposable graph} \Big\}.
$$
Note that if $K$ is vertex decomposable, then $\s(K) \neq \emptyset$.
The following observation is crucial for the
proof of Theorem \ref{scm-bipartite}.
 \begin{obs} \label{reg-obs}
 Let $G$ be a vertex decomposable graph and $x \in \s(G)$. By \cite[Theorem 4.2]{HaWood},
 $$\reg(I(G))=\max\Big\{\reg(I(G \setminus x)), ~\reg(I(G \setminus N_G[x]))+1\Big\}.$$
 Therefore, $\reg(I(G \setminus N_G[x]))+1 \leq \reg(I(G))$.
 \end{obs}

We prove Conjecture \ref{ABBH-conj} for the class of vertex decomposable graphs.
Since induced
subgraphs of a vertex decomposable graph are not necessarily be
vertex decomposable, we cannot apply Theorem \ref{main-result} to get
the desired inequality. 
However, we can prove Conjecture \ref{ABBH-conj} for the class of vertex decomposable graphs
almost verbatim of the proof of Theorem \ref{main-result} and we sketch the 
proof with the same notation as in the proof of Theorem \ref{main-result}.

\begin{theorem}\label{scm-bipartite}\mbox{}
If $G$ is a vertex decomposable graph, then for all  $q \geq 1$,
\[
 \reg(I(G)^q) \leq 2q+\reg(I(G))-2.
\]
\end{theorem}
\begin{proof}
By applying \cite[Theorem 5.2]{banerjee} and using induction on $q$, it is
enough to prove that for any $s \geq 0$ and any minimal generator $M$
of $I(G)^s, ~\reg(I(G)^{s+1}:M) \leq \reg(I(G)).$ 
Let $G'$ be the graph associated
to the ideal $\widetilde{(I(G)^{s+1}:e_1\cdots e_s)}$ which is
contained in an appropriate polynomial ring $R_1$, where $e_1, \ldots, e_s \in E(G)$. We prove that
$\reg(I(G')) \leq \reg(I(G))$ by induction on $s+|V(G)|$ which completes the proof. 
When either $s=0$ or $|V(G)|=2$, we have $G'=G$ and the assertion is clear.
Now, assume that $s \geq 1$ and $|V(G)| > 2$.
Let $e_i = \{a_i, b_i\}$. If
$\deg_G(a_i) = 1$ or $\deg_G(b_i) = 1$ for some $i$, then the
assertion follows as in the proof of Theorem \ref{main-result}.
Therefore, we may assume that $\deg_G(a_i) \geq 2$ and $\deg_G(b_i)
\geq 2$ for all $1 \leq i \leq s$.

 \vskip 1mm \noindent
\textsc{Case 1:} Suppose $e_i \cap \s(G) \neq \emptyset$, for some $1
\leq i \leq s$. 
\vskip 1mm \noindent
Without loss of generality, assume that $e_s \cap \s(G) \neq \emptyset$ and $a_s \in \s(G)$.
Now proceeding as in the proof Theorem \ref{main-result} with the same
notation, one gets 
\begin{eqnarray*}
\reg(J+(\X_{b_s}))  &\leq & \reg(I((G \setminus N_{G}[b_s])'))\leq \reg(I(G \setminus N_G[b_s])) \leq
  \reg(I(G));  \\  
  \reg(J : y_i) & \leq & \reg(I((G \setminus N_G[y_i,a_s])')) \leq \reg(I((G \setminus N_G[a_s])')) \\
  & \leq&  \reg(I(G \setminus N_G[a_s])) < \reg(I(G)).
\end{eqnarray*}
Here, we use  Lemmas  \ref{ind-reg}, \ref{even-lemma}
and inductive hypothesis (since $G \setminus N_G[b_s]$ and $G \setminus N_G[a_s]$
are vertex decomposable graphs) along with Observation \ref{reg-obs}
for the above conclusions.

Using these inequalities, we conclude, as in the proof
of Theorem \ref{main-result}, that $\reg(J) \leq \reg(I(G))$.

\vskip 1mm \noindent
\textsc{Case 2:}
Suppose $e_i \cap \s(G) = \emptyset$, for all $1 \leq i \leq s$.
Let $x \in \s(G)$.  By \cite[Theorem 3.4]{Ha2}, 
$\reg(I(G')) \leq \max
\Big\{\reg(I(G' \setminus x)), \reg(I(G' \setminus
N_{G'}[x])+1\Big\}.$
Since $G \setminus x$ is vertex decomposable, we can use Lemmas
\ref{single_vertex}, \ref{tech_lemma} along with Observation
\ref{reg-obs} to derive $\reg(I(G')) \leq \reg(I(G))$ as done in
the proof of \textsc{Case 2} in Theorem \ref{main-result}.


This completes the proof.
\end{proof}

As an immediate consequence of the above result, we obtain the linear
polynomial corresponding to $\reg(I(G)^q)$ for several 
classes of graphs.
A \textit{simplicial vertex} of a graph $G$ is a vertex $x$ such that the neighbors of $x$ form a complete subgraph in $G$.

\begin{corollary}\label{main-cor} 
Let $G$ be a graph with one of the following properties:
\begin{enumerate}
 \item vertex decomposable and contains no 5-cycles;
 \item vertex decomposable and contains no induced 5-cycles and 4-cycles;
 \item for any independent set $A$, the graph $G \setminus N_G[A]$ is a collection
 of isolated vertices or has a simplicial vertex of degree at least one;
 \item  sequentially Cohen-Macaulay bipartite;
 \item  obtained from a graph $H$ by 
 adding whiskers on a subset $S$ of the vertex set of $H$ such that $H \setminus S$ is 
   chordal.
\end{enumerate}
Then for all $q \geq 1$, $$\reg(I(G)^q)=2q+\nu(G)-1.$$ 
\end{corollary}
\begin{proof} (1) \& (2): From \cite[Theorem 2.4]{khosh_moradi} and
\cite[Theorem 24]{BC}, it follows that for these classes of graph,
$\reg(I(G)) = \nu(G) + 1$. Hence the assertion follows from
\cite[Theorem 4.5]{selvi_ha} and Theorem \ref{scm-bipartite}.
\vskip 1mm \noindent
(3) First we prove that $\reg(I(G))\leq 
\nu(G)+1$, by induction on $|V(G)|$. Since, by our
assumption, $G$ does not have isolated vertices, the assertion is
immediate for the base case $|V(G)| = 2$. Assume now that $|V(G)| >
2$.
Since the empty set is independent, there is a simplicial vertex $x$ in $G$.
Let $N_G(x)=\{x_1,\ldots,x_m\}$ with $m \geq 1$. 
Set $I(G)=I_0$ and $I_l=I_0+(x_1,\ldots,x_l)$ for $1 \leq l \leq m$. Then, for $0 \leq l \leq m-1$, we have
\begin{equation*}\label{exact1}
 0 \longrightarrow \frac{R}{(I_l:x_{l+1})}(-1)
 \overset{\cdot x_{l+1}}{\longrightarrow} \frac{R}{I_l} \longrightarrow \frac{R}{I_l+(x_{l+1})}
 \longrightarrow 0.
\end{equation*}
Therefore,
\[
 \reg(I_0) \leq \max \{ \reg(I_m), \reg(I_l:x_{l+1})+1: 0 \leq l \leq m-1\}.
\]

Since $\{x_1, \ldots, x_l\} \subset N_G[x_{l+1}]$, 
$((G \setminus N_G[x_{l+1}])\setminus \{x_1,\ldots,x_l\})=G \setminus N_G[x_{l+1}] \text{ for any } 
0 \leq l \leq m-1.$ Hence $(I_l : x_{l+1}) = (I_0 : x_{l+1})$.
It is easy to see that
for any vertex $u$ of $G$ and an independent set $B$ of $G \setminus
N_G[u]$, the set $B \cup \{u\}$ is an independent set of $G$.
Then $G\setminus N_G[u]$ is also a graph with the property (3).
Therefore, we may apply inductive hypothesis to get
$$\reg(I_m)=\reg(I(G \setminus N_G[x])) \leq
 \nu(G \setminus N_G[x])+1 \leq \nu(G)+1$$
 and for any  $0 \leq l \leq m-1,$
 $$\reg(I_l:x_{l+1})= \reg(I_0:x_{l+1})=\reg(I(G \setminus N_G[x_{l+1}])) \leq \nu(G \setminus N_G[x_{l+1}])+1.$$

 If $\{f_1,\ldots,f_t\}$ is an induced matching of $G \setminus N_G[x_{l+1}]$,
 then $\{f_1,\ldots,f_t,\{x,x_{l+1}\}\}$ is an induced matching of $G$.
 Therefore $\nu(G \setminus N_G[x_{l+1}])+1 \leq \nu(G)$. Hence
 $ \reg(I(G))\leq \nu(G)+1.$
 By \cite[Corollary 5.5]{Russ11}, $G$ is a vertex decomposable graph. Therefore, by
 Theorem \ref{scm-bipartite} and \cite[Theorem 4.5]{selvi_ha},
 $\reg(I(G)^q)=2q+\nu(G)-1.$

 \vskip 1mm \noindent
 (4) By \cite[Theorem 2.10]{adam}, $G$ is vertex decomposable. 
  Since a bipartite graph  contains no 5-cycles, the assertion follows from
  (1).
  \vskip 1mm \noindent
 (5) In \cite{Dirac61}, Dirac proved that a graph $L$ is chordal if and only if every induced subgraph of 
 $L$ has a simplicial vertex. 
 Let $A$ be any independent set of $G$. If $S \subseteq N_G[A]$, then
 $G\setminus N_G[A]$ is chordal or isolated vertices.
 If $S\varsubsetneq N_G[A]$, then $G \setminus N_G[A]$ has atleast one 
 whisker and hence $G \setminus N_G[A]$ has a simplicial vertex. 
 Therefore $G \setminus N_G[A]$ has a simplicial vertex for any independent
 set $A$. Hence, by (3), $\reg(I(G)^q)=2q+\nu(G)-1$ for all $q \geq 1$.
\end{proof}

\vspace*{2mm} \noindent
\textbf{Acknowledgement:} We would like to thank  Fahimeh Khosh-Ahang,
Huy T{\`a}i H\`a, Adam Van Tuyl and Russ Woodroofe for several
clarifications on our doubts on their results. We also would like to
thank Arindam Banerjee and Huy T\`ai H\`a for pointing out an error in
one of the proofs in an earlier version of the manuscript. We
extensively used \textsc{Macaulay2}, \cite{M2}, \textsc{SAGE},
\cite{sage} and the package \textsc{SimplicialDecomposability},
\cite{Cook}, for testing our computations. 
The second author is partially supported by DST, Govt of India under the DST-INSPIRE
Faculty Scheme.
We would also like to express our sincere gratitude to anonymous
referees for meticulous reading and suggesting several improvements.

\bibliographystyle{abbrv}  
\bibliography{refs_reg}
\end{document}